\documentclass[11pt]{amsart}
\usepackage[english]{babel}
\usepackage[utf8]{inputenc}

\usepackage[inner=1.1in, outer=1.1in, bottom =2.9cm, top =2.9cm]{geometry}

\makeatletter
\newcommand*\Cdot{\mathpalette\Cdot@{.5}}
\newcommand*\Cdot@[2]{\mathbin{\vcenter{\hbox{\scalebox{#2}{$\m@th#1\bullet$}}}}}
\makeatother

\usepackage{amsmath, amssymb, amsfonts, amsthm}
\usepackage{graphicx, overpic}
\usepackage{mathrsfs}
\usepackage{mathtools}
\usepackage[usenames,dvipsnames,svgnames,table]{xcolor}
\usepackage{enumerate}
\usepackage{enumitem}
\usepackage{shuffle} 
\usepackage{lipsum}
\usepackage{color}
\usepackage{phoenician}
\usepackage{xkeyval}
\usepackage{mathrsfs}
\usepackage{tikz}
\usepackage{pst-node} 
\usepackage{tikz-cd} 
\usepackage[OT2,T1]{fontenc} 
\usepackage{thm-restate} 
\usepackage{float}
\usepackage{eucal}
\usepackage{microtype}
\usepackage[strict]{changepage} 
\usepackage{esvect}
\usepackage{etoolbox}
\usepackage{wrapfig}
\usepackage{caption}
\usepackage{cite}
\usepackage{mathrsfs}
\DeclareMathAlphabet{\mathpzc}{OT1}{pzc}{m}{it}
\usepackage{bbding}
\usepackage{array}
\usepackage{makecell}
\usepackage{stmaryrd}
\usepackage{lmodern}
\usepackage{amsbsy}
\usepackage{esvect}
\usepackage{bbding}
\usepackage{thm-restate}
\usepackage[toc,page, titletoc,title]{appendix}
\graphicspath{ {figures/} }
\usepackage{url}
\usepackage[pagebackref]{hyperref} 

\makeatletter
\providecommand*{\twoheadrightarrowfill@}{%
  \arrowfill@\relbar\relbar\twoheadrightarrow
}
\providecommand*{\twoheadleftarrowfill@}{%
  \arrowfill@\twoheadleftarrow\relbar\relbar
}
\providecommand*{\xtwoheadrightarrow}[2][]{%
  \ext@arrow 0579\twoheadrightarrowfill@{#1}{#2}%
}
\providecommand*{\xtwoheadleftarrow}[2][]{%
  \ext@arrow 5097\twoheadleftarrowfill@{#1}{#2}%
}
\makeatother

\makeatletter
\newcommand*{\relrelbarsep}{.386ex}
\newcommand*{\relrelbar}{%
  \mathrel{%
    \mathpalette\@relrelbar\relrelbarsep
  }%
}
\newcommand*{\@relrelbar}[2]{%
  \raise#2\hbox to 0pt{$\m@th#1\relbar$\hss}%
  \lower#2\hbox{$\m@th#1\relbar$}%
}
\providecommand*{\rightrightarrowsfill@}{%
  \arrowfill@\relrelbar\relrelbar\rightrightarrows
}
\providecommand*{\leftleftarrowsfill@}{%
  \arrowfill@\leftleftarrows\relrelbar\relrelbar
}
\providecommand*{\xrightrightarrows}[2][]{%
  \ext@arrow 0359\rightrightarrowsfill@{#1}{#2}%
}
\providecommand*{\xleftleftarrows}[2][]{%
  \ext@arrow 3095\leftleftarrowsfill@{#1}{#2}%
}
\makeatother

\makeatletter
\newcommand{\colim@}[2]{%
  \vtop{\m@th\ialign{##\cr
    \hfil$#1\operator@font colim$\hfil\cr
    \noalign{\nointerlineskip\kern1.5\ex@}#2\cr
    \noalign{\nointerlineskip\kern-\ex@}\cr}}%
}
\newcommand{\colim}{%
  \mathop{\mathpalette\colim@{\rightarrowfill@\scriptscriptstyle}}\nmlimits@
}
\renewcommand{\varprojlim}{%
  \mathop{\mathpalette\varlim@{\leftarrowfill@\scriptscriptstyle}}\nmlimits@
}
\renewcommand{\varinjlim}{%
  \mathop{\mathpalette\varlim@{\rightarrowfill@\scriptscriptstyle}}\nmlimits@
}
\makeatother

\DeclareSymbolFont{cyrletters}{OT2}{wncyr}{m}{n}
\DeclareMathSymbol{\Sh}{\mathalpha}{cyrletters}{"58}

\makeatletter
\newcommand*\bigcdot{\mathpalette\bigcdot@{.5}}
\newcommand*\bigcdot@[2]{\mathbin{\vcenter{\hbox{\scalebox{#2}{$\m@th#1\bullet$}}}}}
\makeatother

\usetikzlibrary{tikzmark,decorations.pathreplacing}
\usetikzlibrary{shapes,shadows,arrows}
\usetikzlibrary{decorations.markings}
\usetikzlibrary{positioning,calc}
\tikzset{near start abs/.style={xshift=1cm}}

\DeclareSymbolFont{symbolsC}{U}{txsyc}{m}{n}
\DeclareMathSymbol{\Searrow}{\mathrel}{symbolsC}{117}

\hypersetup{
	colorlinks=true,
	linkcolor=black,
	filecolor=black,      
	urlcolor=blue,
	citecolor= 	red,
}
\DeclareSymbolFont{extraup}{U}{zavm}{m}{n}
\DeclareMathSymbol{\varheart}{\mathalpha}{extraup}{86}
\DeclareMathSymbol{\vardiamond}{\mathalpha}{extraup}{87}

\theoremstyle{definition}
\newtheorem{thm}{Theorem}[section]
\newtheorem{cor}{Corollary}[thm]
\newtheorem{lem}[thm]{Lemma}
\newtheorem{prop}[thm]{Proposition}
\newtheorem{conj}[thm]{Conjecture}

\theoremstyle{definition}

\newtheorem{ex}{Example}[section]

\newcommand{\gG}{\Gamma}

\newcommand{\gs}{\sigma}

\newcommand{\gd}{\delta}

\newcommand{\go}{\omega}

\newcommand{\bZ}{\mathbb{Z}}

\newcommand{\cC}{\CMcal{C}}
\newcommand{\cD}{\CMcal{D}}

\newcommand{\cP}{\CMcal{P}}
\newcommand{\cM}{\CMcal{M}}

\newcommand{\cF}{\CMcal{F}}
\newcommand{\cR}{\CMcal{R}}

\newcommand{\cW}{\CMcal{W}}
\newcommand{\cI}{\CMcal{I}}

\newcommand{\cB}{\CMcal{B}}

\newcommand{\cS}{\CMcal{S}}

\newcommand{\cT}{\CMcal{T}}

\newcommand{\CAT}{\operatorname{CAT}}

\newcommand{\FG}{\operatorname{FG}}

\newcommand{\Aut}{\operatorname{Aut}}
\newcommand{\PG}{\operatorname{PG}}
\newcommand{\Cir}{\operatorname{Cyc}}

\newcommand{\la}{\langle}
\newcommand{\ra}{\rangle}

\newcommand{\wt}{\widetilde}

\definecolor{Red}{rgb}{0.8,0,0.2}

\newcommand{\GG}[1]{}

\makeatletter
\def\@footnotecolor{red}
\define@key{Hyp}{footnotecolor}{%
 \HyColor@HyperrefColor{#1}\@footnotecolor%
}
\def\@footnotemark{%
    \leavevmode
    \ifhmode\edef\@x@sf{\the\spacefactor}\nobreak\fi
    \stepcounter{Hfootnote}%
    \global\let\Hy@saved@currentHref\@currentHref
    \hyper@makecurrent{Hfootnote}%
    \global\let\Hy@footnote@currentHref\@currentHref
    \global\let\@currentHref\Hy@saved@currentHref
    \hyper@linkstart{footnote}{\Hy@footnote@currentHref}%
    \@makefnmark
    \hyper@linkend
    \ifhmode\spacefactor\@x@sf\fi
    \relax
  }%
\makeatother

\hypersetup{footnotecolor=blue}

\makeatletter
\patchcmd{\@startsection}
  {\@afterindenttrue}
  {\@afterindentfalse}
  {}{}
\makeatother

\title{Singer Cyclic Lattices of Type M}
\author{William Norledge}
\address[William Norledge]{Pennsylvania State University}
\email{wxn39@psu.edu}

\begin{document}

\renewcommand{\chapterautorefname}{Chapter}
\renewcommand{\sectionautorefname}{Section}
\renewcommand{\subsectionautorefname}{Section}

\begin{abstract}
We construct the type-preserving panel-regular lattices in buildings of type M, for M with entries two, three or infinity, which have cyclic stabilizers of spherical 2-residues. We obtain these lattices as fundamental groups of their associated quotient buildings, which are constructed by amalgamating quotients of generalized digons and triangles by actions of Singer cyclic groups. Each quotient construction has an associated gluing matrix which encodes how the quotient generalized polygons are amalgamated. We show that these gluing matrices also encode presentations of the panel-regular lattices.  
\end{abstract}

\maketitle

\setcounter{tocdepth}{1} 
\hypertarget{foo}{ }
\tableofcontents

\section*{Introduction} 

A Singer lattice is a discrete subgroup of the automorphism group of a locally finite building whose associated quotient is finite and which acts regularly on panels. See \cite{anneprobsonpolycomp} for details on lattices in buildings. A Singer cyclic lattice is a Singer lattice which has cyclic stabilizers of spherical $2$-residues. In this paper, we construct all the Singer cyclic lattices of type $M$ in the case where the off-diagonal entries of $M$ are in $\{ 2,3,\infty\}$, and the defining graph of $M$ is connected. Thus, via the Davis realization of a building \cite{davis94buildingscat0}, we construct examples of lattices in CAT(0) polyhedral complexes with links complete bipartite graphs and (incidence graphs of) finite projective planes. We first describe the $2$-residues which can appear in the quotient building of a Singer cyclic lattice of type $M$, and then determine all the ways in which these $2$-residues can be amalgamated to form a quotient. Our lattices generalize those in \cite{essert2013geometric}, where the Singer lattices of type $\wt{A}_2$ were constructed using complexes of groups. These lattices were further studied in \cite{witzel2016panel}. 

We obtain presentations of our lattices which are free products of a collection of Singer cycles quotiented by a set of relations which are read off the (decorated) defining graph of $M$ by going around cycles (\autoref{mainmain}). To construct quotients, we make use of the theory of Weyl graphs developed in \cite{nor2}. Weyl graphs are certain directed multigraphs with extra structure that can be used to model quotients of buildings by type-preserving chamber-free group actions. Unlike complexes of groups, Weyl graphs make direct use of the combinatorial $W$-metric structure enjoyed by buildings, and are in the spirit of \cite{tits81local}, \cite[Chapter 4]{ronanlectures}.

\subsection*{Structure} 

We begin in \autoref{sec:Preliminaries} by giving some basic definitions, and recalling aspects of the theory of Weyl graphs from \cite{nor2}. In \autoref{sec:Singer Cyclic Polygons} we obtain realizations of those Weyl graphs of type $A_2$ (triangles) and $A_1\times A_1$ (digons) which can exist as $2$-residues in the quotient of a Singer cyclic lattice. In \autoref{sec:Singer Cyclic Lattices of Type $M$} we construct all the Singer cyclic lattices of type $M$, for $M$ with values in $\{ 2,3,\infty\}$. We obtain group presentations of the Singer cyclic lattices. We finish in \autoref{sec:Examples of Singer Cyclic Lattices} by giving some explicit examples. 

\subsection*{Acknowledgments.} 

This research was partly supported by a grant from Templeton Religion Trust as part of the mathematical picture language project at Harvard University. We also thank Newcastle University for their support. 

\section{Preliminaries}  \label{sec:Preliminaries}

We recall some preliminary definitions and structures.

\subsection{Weyl graphs}\label{sec:Weyl graphs}

We recall Weyl graphs from \cite{nor2}. Let 
\[
M:S\times S\to \bZ_{\geq 1}\cup \{\infty\}
,\qquad
(s,t)\mapsto m_{st}
\]
be a Coxeter matrix,\footnote{ symmetric, and $m_{st}$=1 if and only if $s=t$} and let 
\[
W:=\big \la S\ |\ (st)^{m_{st}}=1 : s,t\in S  \big  \ra 
\] 
be the Coxeter group of $M$. See \cite[Section 2]{ab08} for details on Coxeter groups. A \emph{simplicial} graph $L=(V(L),E(L))$ is an undirected graph without loops or multiple edges in which the \emph{edges} $E(L)$ are modeled as $2$-element subsets of the \emph{vertices} $V(L)$. The \emph{defining graph} $L=L(M)$ of $M$ is the edge labeled simplicial graph with
\[V(L):=S\qquad \text{and} \qquad E(L):=\big\{\{s,t\}: s,t\in S,\  s\neq t,\ m_{st}< \infty \big\} \]
where the edge $\{s,t\}\in E(L)$ is labeled by $m_{st}$. 

Let a \emph{graph} $\cW$ of type $M$ consist of a set of \emph{chambers} $\cW_0$ and a set of \emph{edges} $\cW_1$ which are equipped with functions
\[     
\cW_1\to \cW_0\times \cW_0
,\qquad  
i\mapsto (\iota(i),\tau(i)) 
\qquad \text{and} \qquad  
\cW_1\to S
,\qquad   
i\mapsto \upsilon(i) 
.\]
Then $\iota(i)$ and $\tau(i)$ are called the \emph{initial} and \emph{terminal} chambers of $i$ respectively, and $\upsilon(i)\in S$ is called the \emph{type} of $i$. A \emph{morphism} $\go$ of graphs of type $M$ consists of a function of the chambers $\go_0$ and a function of the edges $\go_1$ which preserves initial/terminal chambers and types. A \emph{gallery} is a morphism $\beta:\lfloor a,b \rfloor \to \cW$ where $\lfloor a,b \rfloor$ has underlying graph
\[
\lfloor a,b \rfloor \ = \ a \xrightarrow{i_{a+1}}  a+1 \xrightarrow{i_{a+2}} \quad  \cdots \quad   \xrightarrow{i_{b-1}} b-1  \xrightarrow{i_{b}} b  , \qquad a,b \in \bZ,\quad  a\leq b
.\]
The \emph{type} $\beta_S$ of a gallery $\beta$ is the word over $S$ given by 
\[
\beta_S:= \upsilon(i_{a+1})\dots  \upsilon(i_{b})
.\]
The \emph{$W$-length} $\beta_W$ of a gallery $\beta$ is the element of $W$ obtained by treating its type $\beta_S$ as a product of generators of $W$. A \emph{geodesic} $\gamma:\lfloor a,b \rfloor \to \cW$ is a gallery whose type has minimal length amongst all the expressions of its $W$-length as a product of generators. For $s,t\in S$ with $s\neq t$ and $m_{st}< \infty$, an \emph{$(s,t)$-geodesic} $\rho_m(s,t)$, or \emph{alternating geodesic}, is a geodesic with type
\[ 
p_m(s,t):= \underbrace{stst\dots}_{\text{$m$ letters}} \qquad \qquad  \text{(it follows that $m\leq m_{st}$)}
.\] 
A \emph{maximal $(s,t)$-geodesic} $\rho(s,t)$, or \emph{maximal alternating geodesic}, is a gallery with type 
\[
p(s,t):=\underbrace{stst\dots}_{\text{$m_{st}$ letters}}
.\] 
An \emph{$(s,t)$-cycle} $\theta(s,t)$ is gallery with $\beta(a)=\beta(b)$ and type $p_{2m_{st}}(s,t)$. Let a \emph{Weyl graph} of type $M$, also denoted $\cW$, consist of a graph of type $M$ equipped with the following extra data,
\begin{enumerate}
\item 
for each $s\in S$, we have a (small) groupoid $\cW_s$ with objects the chambers of $\cW$, and whose non-identity morphisms are those edges of $\cW$ with type $s$, called the \emph{panel groupoid} of type $s$
\item
for each choice of $s,t\in S$ with $s\neq t$ and $m_{st}< \infty$, we have a set $\cW(s,t)$ of $(s,t)$-cycles in $\cW$ called \emph{defining $(s,t)$-suites}, or \emph{defining suites}. 
\end{enumerate}
The \emph{fundamental groupoid} $\overline{\cW}$ of $\cW$ is the freest groupoid generated by the $\cW_s$, $s\in S$, such that the composition of edges of any defining suite is an identity \cite[Section 2.6]{nor2}. For the data $\cW$ to be a Weyl graph, we require
\begin{enumerate} [itemindent=0cm, leftmargin=1.9cm]
\item[\textbf{(PW0)}] the $\cW_s$, $s\in S$, are non-trivial\footnote{ the trivial groupoid is the unique groupoid with one object and one morphism} 
\item[\textbf{(PW1)}] a morphism in $\overline{\cW}$ which is the composition of edges of a maximal $(s,t)$-geodesic is also equal to the composition of edges of a maximal $(t,s)$-geodesic
\item[\textbf{(W)}] any two geodesics whose composition of edges are equal in $\overline{\cW}$ have the same \hbox{$W$-length}. 
\end{enumerate}
See \cite[Section 3.1]{nor2}. In particular, property (W) means that $W$-length factors through to give a well-defined \emph{metrization} of the fundamental groupoid,
\[
\delta:\overline{\cW}\to W.\] 
See \cite[Section 5.3]{nor2}. Buildings $\Delta$ are recovered as those Weyl graphs $\cW$ such that $\overline{\cW}$ is connected and simply connected. For galleries $\beta$ and $\beta'$ in $\cW$, if their composition of edges coincide in $\overline{\cW}$ then we say that $\beta$ and $\beta'$ are \emph{homotopic}, and write $\beta\sim \beta'$.  

To every type-preserving chamber-free action of a group $G$ on a building $\Delta$ is associated the quotient Weyl graph $G\backslash \Delta$,
\[
\Delta \to G\backslash \Delta.
\] 
Conversely, to every Weyl graph $\cW$ is associated the universal covering building 
\[
\wt{\cW}\to \cW
\] 
equipped with an action of the fundamental group of $\cW$ (which, if we use basepoints, is just a local group of $\overline{\cW}$). See \cite[Section 4.3]{nor2}. These constructions are mutually inverse. Thus, there is essentially a 1-1 correspondence between type-preserving chamber-free actions and Weyl graphs. 

\subsection{Singer Lattices and Singer Graphs}

We call an action of a group on a building \emph{panel-regular} if the induced action on the set of panels of each type is regular, i.e. free and transitive. 
Let us define a \emph{Singer graph} to be a Weyl graph which is the quotient of a panel-regular action. Equivalently, a Singer graph is a Weyl graph whose panel groupoids are all isomorphic to a fixed connected setoid. The cardinal $q$ such that a Singer graph $\cW$ has $k=q+1$-many chambers is called the \emph{order} of $\cW$. We define a \emph{Singer building} of order $q$ to be a building which is the universal cover of a Singer graph of order $q$. A Singer building is locally finite if and only if it has finite order since each of its panels is $k$-many equivalent chambers. Therefore we may define a \emph{Singer lattice} 
\[
\gG< \Aut(\Delta)
\] 
of order $q$ to be a subgroup of the automorphism group of a building $\Delta$ such that $\gG\backslash \Delta$ is a Singer graph of order $q$, where $q$ is finite. We define a \emph{Singer cyclic lattice} to be a Singer lattice which has cyclic stabilizers of spherical $2$-residues. Our definitions of Singer lattices and Singer cyclic lattices generalize those in \cite{essert2013geometric}, \cite{witzel2016panel} to all types of building. 

A \emph{Weyl $m$-gon} is a connected Weyl graph of type $I_2(m)$, a \emph{generalized $m$-gon} is a building of type $I_2(m)$, a \emph{Singer $m$-gon} is a Singer graph of type $I_2(m)$,\footnote{ this is more restrictive than the usual definition of a Singer polygon, many authors define a Singer polygon to be a generalized polygon equipped with an action of a group which is point-regular, e.g. \cite{de2009singer}; our definition of a Singer polygon is equivalent to a generalized polygon equipped with an action of a group which is point-regular and line-regular} and a \emph{Singer cyclic $m$-gon} is a Singer $m$-gon whose fundamental group is cyclic. For $m=2,3$ we say \emph{digon} and \emph{triangle} respectively, and for arbitrary $m$ we say \emph{polygon}. 

A Weyl graph is a Singer graph if and only if its $2$-residues are Singer polygons, and a Singer cyclic lattice is equivalently a Singer lattice whose quotient's $2$-residues are Singer cyclic polygons. These statements follow from covering theory developed in \cite[Section 4]{nor2}. 

\section{Singer Cyclic Polygons} \label{sec:Singer Cyclic Polygons}

In this section we obtain realizations of those Weyl digons and triangles which can exist as $2$-residues in the quotient of a Singer cyclic lattice. For digons the construction is straightforward. For triangles we follow \cite{essert2013geometric}, and use the method of difference sets from finite geometry. We will see that for $q\geq 2$, there exists a unique Singer cyclic digon of order $q$, and for $q\geq 2$ a prime power, there exists a Singer cyclic triangle of order $q$. The uniqueness of this Singer triangle is (equivalent to) a long standing conjecture. 

\subsection{Singer Cyclic Digons}

The isomorphism classes of finite generalized digons are in bijection with pairs $(q_1,q_2)$, for $q_1,q_2\in \bZ_{\geq 1}$. The simplicial building of the digon corresponding to $(q_1,q_2)$ is the complete bipartite graph on $q_1+1$ white vertices and $q_2+1$ black vertices. 

The Weyl graph of the generalized digon corresponding to $(q_1,q_2)$, denoted $\mathbf{D}(q_1,q_2)$, may be constructed as follows. Let
\[
k_1:=q_1+1\qquad  \text{and} \qquad k_2:=q_2+1
.\] 
Let the set of chambers of $\mathbf{D}(q_1,q_2)$ be the set
\[
\mathbf{D}(q_1,q_2)_0:=\bZ/k_1\bZ\times \bZ/k_2\bZ
.\] 
Let $S=\{s,t\}$ be the set of labels of $\mathbf{D}(q_1,q_2)$. The panel groupoid of type $s$ of $\mathbf{D}(q_1,q_2)$ is the setoid corresponding to the equivalence relation
\[ 
(x,y)\sim_s (x',y')\qquad  \text{if}\quad  x=x'
.\]
The panel groupoid of type $t$ of $\mathbf{D}(q_1,q_2)$ is the setoid corresponding to the equivalence relation
\[ 
(x,y)\sim_t (x',y')\qquad \text{if}\quad y=y'
.\]
Thus, the generalized digon $\mathbf{D}(q_1,q_2)$ is a $k_1\times k_2$ grid of chambers, with chambers in the same column being $s$-equivalent, and chambers in the same row being $t$-equivalent. We abbreviate $\mathbf{D}(q):=\mathbf{D}(q,q)$. See the left part of \autoref{fig:3times3}, which shows $\mathbf{D}(2)$.

We now construct the Singer cyclic digons. If a group $G$ acts panel-regularly on $\mathbf{D}(q_1,q_2)$, then $|G|=q_1+1=q_2+1$. Put $q:=q_1=q_2$ and $k:=q+1$. 

\begin{prop} \label{prop:uniquedigon}
Let $G$ be the cyclic group of order $k$. Then $G$ acts panel-regularly on $\mathbf{D}(q)$, and this action is unique up to equivariant automorphism.
\end{prop}

\begin{proof}
The group $G$ acts panel-regularly on $\mathbf{D}(q)$ as follows; pick a generator $g\in G$ and for $(x,y)\in \bZ/k\bZ\times \bZ/k\bZ$ a chamber of $\mathbf{D}(q)$, put
\[g \cdot(x,y)=(x+1,y+1).\] 
This is an automorphism of $\mathbf{D}(q)$ since it is a permutation of the chambers which preserves $\sim_s$ and $\sim_t$. Suppose that $G$ acts in a second way on $\mathbf{D}(q)$, which we denote by `$\ast$'. For $x\in \bZ/k\bZ$, let $x_s$ be the $s$-panel which contains $g^x\ast (0,0)$, and for $y\in \bZ/k\bZ$, let $y_t$ be the $t$-panel which contains $g^y\ast (0,0)$. Let $(x_s,y_t)$ denote the unique chamber which is contained in both $x_s$ and $x_t$. We claim that $(x,y)\mapsto (x_s,y_t)$ is a permutation of the chambers of $\mathbf{D}(q)$. To see this, suppose that
\[(x,y)\mapsto (x_s,y_t),\ \ \ \ (x',y')\mapsto (x'_s,y'_t),   \ \ \ \ \text{and} \ \ \ \   (x_s,y_t)=(x'_s,y'_t).       \]
Then $x_s=x'_s$ and $y_t=y'_t$, and so $x=x'$ and $y=y'$ since the action `$\ast$' is free on panels. This permutation is an automorphism since if $(x,y)\sim_s (x',y')$, then $x=x'$, and so $x_s=x'_s$ (likewise for $t$). To see that this automorphism is equivariant, let $h\cdot x$ and $h\cdot y$ be the integers such that $(h\cdot x,h\cdot y)=h\cdot(x,y)$, and let $h\ast x_s$ and $h\ast y_t$ be the integers such that $(h\ast x_s,h\ast y_t)=h\ast(x_s,y_t)$. For $h=g^n$, we have 
\[(h\cdot x)_s=(x+n)_s=  g^{x+n}\ast (0,0) =g^n  \ast   g^x\ast (0,0) =g^n \ast x_s=h\ast x_s  \]
\[(h\cdot y)_t=(y+n)_s=  g^{y+n}\ast (0,0) =g^n  \ast   g^y\ast (0,0) =g^n  \ast y_t=h\ast y_s. \qedhere  \] 
\end{proof}

\begin{cor}  \label{corcorcor}
For each $q\in \bZ_{\geq 1}$, there is a unique Singer cyclic digon of order $q$ (up to isomorphism). 
\end{cor}

\begin{proof}
Equivariant actions will produce isomorphic quotients. Therefore the result follows from \autoref{prop:uniquedigon}.
\end{proof}

We denote by $k\backslash \mathbf{D}(q)$ the unique Singer cyclic digon of order $q$. We wish to obtain a canonical realization of $k\backslash \mathbf{D}(q)$. Let $G$ be the cyclic group of order $k$, and let $g\in G$ be a generator. We represent $k\backslash \mathbf{D}(q)$ as the quotient of $\mathbf{D}(q)$ by the action of $G$ given by 
\[
g\cdot (x,y)=(x+1,y+1)\] 
for $(x,y)$ a chamber of $\mathbf{D}(q)$. Recall that we have an associated covering 
\[
\pi:\mathbf{D}(q)\to G \backslash \mathbf{D}(q)
.\] 
Let the set of chambers of $k\backslash \mathbf{D}(q)$ be $\cC:=\bZ /k\bZ$, and identify $k\backslash \mathbf{D}(q)$ with $G \backslash \mathbf{D}(q)$ by letting a chamber $x\in \cC$ be the $\pi$-image of $(0,x)\in \mathbf{D}(q)$. To complete our realization of $k\backslash \mathbf{D}(q)$, we need to specify a set of defining suites. 

Given a chamber $x\in \cC$ of $k\backslash \mathbf{D}(q)$, the \emph{flower} $\cF_J(x)$ of $k\backslash \mathbf{D}(q)$ based at $x$ is the set of maximal $(s,t)$-geodesics and maximal $(t,s)$-geodesics which issue from $x$. The suites formed from these geodesics are a set of defining suites for $k\backslash \mathbf{D}(q)$ \cite[Theorem 6.2]{nor3}. A \emph{petal} is a $2$-element subset of $\cF_J(x)$ consisting of a maximal $(s,t)$-geodesic $\rho(s,t)$, together with the unique maximal $(t,s)$-geodesic $\rho(t,s)$ such that $\rho(s,t)\sim \rho(t,s)$ (recall this notation from \autoref{sec:Weyl graphs}). In the following, for chambers $x,y,z\in \cC$, we denote by
\[   
[x \xrightarrow{s} y\xrightarrow{t} z]            
\]
the maximal $(s,t)$-geodesic whose first edge is from $x$ to $y$, and whose second edge is from $y$ to $z$. 


\begin{prop} \label{flowers}
Given a chamber $x\in \cC$, the petals of the flower of $k\backslash \mathbf{D}(q)$ based at $x$ are
\[ 
[x \xrightarrow{s} y\xrightarrow{t} z] \sim   [  x\xrightarrow{t} y'\xrightarrow{s} z  ]   
\] 
 where $y,z,y'\in\cC$, $x\neq y\neq z$, and $y'=x-y+z$. 
\end{prop}

\begin{proof}
Let,
\[
\rho(s,t)=[x\xrightarrow{s} y \xrightarrow{t} z]
\] 
be an maximal $(s,t)$-geodesic of $k\backslash \mathbf{D}(q)$ which issues from $x$. Lift $\rho(s,t)$ with respect to $\pi$ to a gallery $\tilde{\rho}(s,t)$ of $\mathbf{D}(q)$ (lifts of galleries are discussed in \cite[Section 4.1]{nor2}). Let $(c,d)\in \mathbf{D}(q)$ be the initial chamber of $\tilde{\rho}(s,t)$. It follows from the construction of $\mathbf{D}(q)$ that
\[
\tilde{\rho}(s,t)=\big[(c,d)\xrightarrow{s} (c,d+y-x) \xrightarrow{t} (c+y-z,d+y-x) \big]
.\]
Let
\[
\rho(t,s)=[x\xrightarrow{t} y' \xrightarrow{s} z]
\] 
be the unique maximal $(t,s)$-geodesic of $k\backslash \mathbf{D}(q)$ with $\rho(s,t)\sim\rho(t,s)$. Let $\tilde{\rho}(t,s)$ be the unique lifting of $\rho(t,s)$ to a gallery which issues from $(c,d)$. Then
\[
\tilde{\rho}(t,s)=\big[(c,d)\xrightarrow{t} (c+x-y',d) \xrightarrow{s} (c+x-y',d+z-y')\big]
.\]
 We have $\tilde{\rho}(s,t)\sim \tilde{\rho}(t,s)$ since $\pi$ is a covering, and so
\[     
(c+y-z,d+y-x)  = (c+x-y',d+z-y')    
\] 
 which occurs if and only if
\[  
y'= x-y+z
.\qedhere \]
\end{proof}

We let the defining suites of $k \backslash \mathbf{D}(q)$ be those which are induced by the flower based at $0$. Thus, by \autoref{flowers}, the defining suites of $k \backslash \mathbf{D}(q)$ are the cycles
\begin{equation} \label{eq:2}      \tag{$\vardiamond$}
[0\xrightarrow{s} y\xrightarrow{t} z    \xrightarrow{s} (z-y)  \xrightarrow{t} 0           ]  
\end{equation}
where $y,z\in \cC$ and $0\neq y\neq z$. Notice that there are $q^2$ many defining suites. These defining suites are sufficient to define $k\backslash \mathbf{D}(q)$ by \cite[Theorem 6.2]{nor3}.

\begin{ex}
Let $q=2$. The $2^2=4$ defining suites of $3\backslash \mathbf{D}(2)$ are
\[    [0\xrightarrow{s} 1\xrightarrow{t} 0  \xrightarrow{s} 2\xrightarrow{t} 0  ]  \]
\[    [0\xrightarrow{s} 1\xrightarrow{t} 2  \xrightarrow{s} 1\xrightarrow{t} 0  ]  \]
\[    [0\xrightarrow{s} 2\xrightarrow{t} 0\xrightarrow{s} 1\xrightarrow{t} 0  ]  \]
\[    [0\xrightarrow{s} 2\xrightarrow{t} 1\xrightarrow{s} 2\xrightarrow{t} 0  ] .\]
One can check in \autoref{fig:3times3} that the lifting of these galleries to galleries which issue from either $(0,0)$, $(1,1)$, or $(2,2)$ are cycles in $\mathbf{D}(2)$. 
\end{ex}

\begin{lem} \label{lem:rotatingdigon}
Let $q\in \bZ_{\geq 1}$, $k=q+1$, and let $\cC=\bZ /k\bZ$. Let $r\in \bZ /k\bZ$. Then the map 
\[
\cC\to \cC, \qquad x\mapsto x+r
\] 
is an automorphism of $k \backslash \mathbf{D}(q)$. 
\end{lem}

\begin{proof}
This map is clearly a chamber system automorphism. To check that it is also an automorphism of Weyl data, by the characterization of isomorphisms and the properties of flowers, we just have to check the preservation of the flowers of $k \backslash \mathbf{D}(q)$. This follows from \autoref{flowers}, and the fact that
\[y'+r=(x+r)-(y+r)+(z+r). \qedhere  \]
\end{proof}

\begin{cor} \label{cor:rotatingdigon}
Let $\cP$ be a Singer cyclic digon of order $q$, and let $C\in \cP$ be any chamber. Then there exists an isomorphism $\go:\cP\rightarrow k \backslash \mathbf{D}(q)$ such that $\go(C)=0$.
\end{cor}

\begin{proof}
By \autoref{corcorcor} there exists an isomorphism $\go':\cP\rightarrow k \backslash \mathbf{D}(q)$. Let $\go''$ be the isomorphism
\[\go'': k \backslash \mathbf{D}(q)  \to k \backslash \mathbf{D}(q),   \qquad  x\mapsto x-\go'(C).\] 
Then we can take $\go=\go''  \circ \go'$. 
\end{proof}

\begin{figure}[t]
\centering
\includegraphics[scale=1]{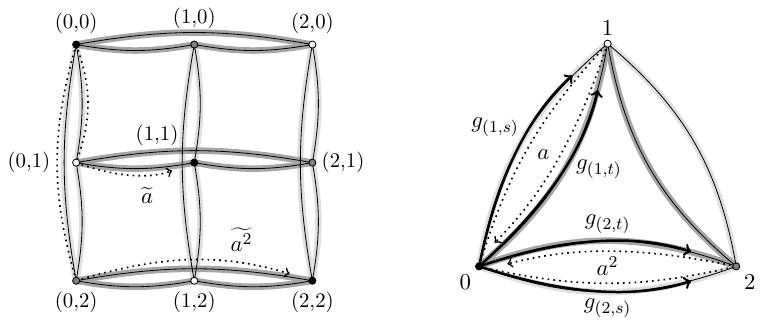}
\caption{The generalized digon $\mathbf{D}(2)$ (left) and the Singer cyclic digon $3\backslash \mathbf{D}(2)$ (right)}
\label{fig:3times3}
\end{figure}  

See \cite[Section 6]{nor2} for details on calculating universal and fundamental groups of Weyl graphs. To calculate the universal group of $k \backslash \mathbf{D}(q)$ we need to equip $k \backslash \mathbf{D}(q)$ with a generating set $\cS=\cB\sqcup \cI$ \cite[Section 6.1]{nor2}. Since each panel groupoid of $k \backslash \mathbf{D}(q)$ is a setoid, we have $\cB=\emptyset$. Let 
\[\cC^{\ast}:=\{1,\dots, q\}\subset \cC.\] 
For $n\in \cC^{\ast}$ and $\sigma\in \{s,t\}$, let $g_{(n,\sigma)}$ be the edge $0\xrightarrow{\sigma}n$ of $k \backslash \mathbf{D}(q)$. Then put
\[\cS=\cI:=\big \{    g_{(n,s)} , g_{(n,t)},g_{(n,s)}^{-1} , g_{(n,t)}^{-1}  :  n\in \cC^{\ast}    \big     \}.\] 
For notational convenience, let $g_{(0,s)}$, $g_{(0,t)}$, $g_{(0,s)}^{-1}$, and $g_{(0,t)}^{-1}$ denote the empty word. Then it follows from (\ref{eq:2}) that the set of $\cS$-suites of $k \backslash \mathbf{D}(q)$ is
\[\cR= \big\{    g_{(y,s)}  g_{(y,t)}^{-1} g_{(z,t)} g_{ (z, s) }^{-1} g_{ (z-y,s)}   g_{(z-y,t)}^{-1}   :\ y,z\in \cC;\ 0\neq y \neq z     \big \}. \]
 We now substitute $a_{(n)}=g_{(n,s)} g_{(n,t)}^{-1}$, for $n\in \cC^{\ast}$. Thus, the new generating set is
\[\cS':= \big \{    g_{(n,s)} , g_{(n,t)},g_{(n,s)}^{-1} , g_{(n,t)}^{-1}, a_{(n)} :  n\in \cC^{\ast}     \big    \}\] 
and, letting $a_{(0)}$ and $a_{(0)}^{-1}$ denote the empty word, a new set of equivalent relations is
\begin{align*}
	\cR':= \big\{\       a_{(y)} a_{(z)}^{-1} a_{(z-y)}=1   &:\ y,z\in \cC;\ 0\neq y \neq z \ \\
	a_{(n)}=g_{(n,s)} g_{(n,t)}^{-1} &:\  n\in \cC^{\ast} \ \big \}. 
\end{align*}
 By putting $z=0$, we see that $a_{(y)} a_{(-y)}=1$ for $y\in \cC^{\ast}$. By putting $z=1$, we see that $a_{(y)}a_{(1)}^{-1} a_{(1-y)}   = 1$, or equivalently $a_{(y)}=a_{(y-1)}a_{(1)}$, for $y\in \cC^{\ast}$, $y\neq 1$.  Thus, by induction, we have $a_{(y)}=a_{(1)}^y$ for $y\in \cC^{\ast}$. In particular, we have $a_{(1)}^{k}=1$. If we put $a=a_{(1)}$ and include the consequence $a^{k}=1$, then the relations $a_{(y)} a_{(z)}^{-1} a_{(z-y)}=1$ are redundant. Therefore a new generating set is
\[\cS'':=\big \{    g_{(n,s)} , g_{(n,t)},g_{(n,s)}^{-1} , g_{(n,t)}^{-1}, a :  n\in \cC^{\ast}    \big     \}\] 
and a new set of equivalent relations is
\[                   \cR'':= \big\{ a^{k}=1,  \     a^n=g_{(n,s)} g_{(n,t)}^{-1} :   n\in \cC^{\ast}    \big \} .  \]
Thus, we obtain the following.

\begin{lem} \label{universalgruop2}
	Let $k \backslash \mathbf{D}(q)$ be the unique Singer cyclic digon of order $q$. Then the universal group of $k \backslash \mathbf{D}(q)$ is
	\[\FG(k \backslash \mathbf{D}(q))=   \big \la      g_{(n,\sigma)} ,g_{(n,\sigma)}^{-1}, a     \ |  \  
	a^{k}=1,  \     a^n=g_{(n,s)} g_{(n,t)}^{-1},\ g_{(n,\sigma)}g_{(n,\sigma)}^{-1}=1        \big \ra  \]
	for $\sigma\in \{s,t\}$ and $n\in \cC^{\ast}$.  
\end{lem}

To calculate the fundamental group of $k \backslash \mathbf{D}(q)$ we need to quotient out a spanning tree \cite[Theorem 6.1]{nor2}. Of course, we know the fundamental group should be cyclic of order $k$. Let $T$ be the spanning tree 
\[
T:=\big \{g_{(n,t)},g_{(n,t)}^{-1}: n\in \cC^{\ast}\big  \}
.\] 
Then we recover the fundamental group of $k\backslash \mathbf{D}(q)$,
\[  
\pi_1(k\backslash \mathbf{D}(q),T)=\big \la  g_{(n,s)}  ,g_{(n,s)}^{-1} , a\ |\  a^{k}=1,  \     a^n=g_{(n,s)}, \  g_{(n,s)}g_{(n,s)}^{-1}=1    \big  \ra\sim  \big \la a| a^k=1    \big  \ra            
\] 
for $n\in \cC^{\ast}$. Notice that the image of the gallery $[ 0\xrightarrow{s} n \xrightarrow{t} 0   ]$ in the fundamental group at $T$ is $a^n$, see \autoref{fig:3times3}. One obtains an action of $\pi_1(k\backslash \mathbf{D}(q),T)$ on $\mathbf{D}(q)$ after picking a chamber in $\mathbf{D}(q)$. In fact, different choices of chamber will result in the same action since $\pi_1(k\backslash \mathbf{D}(q))$ is abelian.

\subsection{Singer Cyclic Triangles}

We now obtain realizations of the Singer cyclic triangles using the method of difference sets. We denote by $\cT(q)$ the generalized triangle whose associated simplicial building is the incidence graph of $\PG(2,q)$.

In the language of Weyl graphs, the method of difference sets $\cD$ in a group $G$ provides a way of constructing a generalized triangle $\cT(\cD)$ and a universal cover
\[
\pi:\cT(\cD)\rightarrow G \backslash \cT(\cD)
\] 
where $\pi$ is the quotient map associated to a panel-regular action of $G$ on $\cT(\cD)$. See \cite{dembowski2012finite} for details on difference sets. We focus on the case where $G$ is cyclic, or equivalently, where $G \backslash \cT(\cD)$ is a Singer cyclic triangle. Such difference sets are usually called cyclic difference sets, and are studied in \cite{berman53finite}. We define a \emph{difference set} $\cD$ of order $q$ to be a subset $\cD\subset \bZ/\gd\bZ$ such that the map
\[
\cD \times \cD \to \bZ/\gd\bZ,\qquad  (x,y)\mapsto x-y
\]
when restricted to the off-diagonal elements is a bijection into $\big\{1,\dots,\gd-1\big\}$. Thus, for all non-zero $n\in \bZ/\gd\bZ$, there exists a unique pair $d,d'\in \cD$ such that $d-d'=n$. Notice that $|\cD| = q + 1=k$. A difference set $\cD$ is called \emph{based} if $0\in  \cD$, in which case we let $\cD^{\ast}:=\cD\setminus \{0\}$. 

Let $\cD$ be a difference set of order $q$, and let $x\in \bZ/\gd \bZ$ and $r \in \Aut(\bZ/\gd \bZ)=\bZ/\gd \bZ^{\ast}$. Then 
\[
r\cD+x=\{rd+x:d\in \cD\}
\] 
is also a difference sets of order $q$. Two difference sets $\cD$ and $\cD'$ are called \emph{equivalent} if there exists $x\in \bZ/\gd \bZ$ and $r\in \bZ/\gd \bZ^{\ast}$ such that $\cD'=r\cD+x$. Our notion of equivalence is different to that of \cite{berman53finite}.

Let $\cD$ be a difference set of order $q$. By results of Singer \cite{singer1938theorem}, we obtain a generalized triangle $\cT(\cD)$ of order $q$ from $\cD$ as follows. Let the chambers of $\cT(\cD)$ be the set
\[
\cT(\cD)_0:=\big \{   (x,x+d) :  x \in \bZ/\gd\bZ, \   d \in \cD  \big    \}          
.\] 
Let $S=\{s,t \}$ be the set of labels of $\cT(\cD)$. The panel groupoid of type $s$ of $\cT(\cD)$ is the setoid corresponding to the equivalence relation
\[
(x,y)\sim_s (x',y')\qquad   \text{if}\qquad   x=x'
.\]
 The panel groupoid of type $t$ of $\cT(\cD)$ is the setoid corresponding to the equivalence relation
\[  
(x,y)\sim_t (x',y')\qquad   \text{if}\qquad  y=y'
.\]
\autoref{fig:fano013} shows $\cT(\cD)$ for $\cD=\{0,1,3\}$, in which case $\cT(\cD)\cong \cT(2)$. Let $G$ be the cyclic group of order $\gd$, and let $g\in G$ be a generator. Then $G$ acts panel-regularly on $\cT(\cD)$ via $g\cdot (x,y)=(x+1,y+1)$. Hence, we obtain a universal cover 
\[\pi:\cT(\cD)\to  G \backslash \cT(\cD)\] 
where $G \backslash \cT(\cD)$ is a Singer cyclic triangle of order $q$. Let us denote by $\gd \backslash \cT(\cD)$ a Weyl graph which is isomorphic to $G \backslash \cT(\cD)$.

A difference set $\cD$ is called \emph{Desarguesian} if $\cT(\cD)\cong \cT(q)$. We have the following well-known open conjecture.

\begin{conj} \label{conj}
	All difference sets are Desarguesian.
\end{conj}  

Let $\cD$ and $\cD'$ be equivalent difference sets of order $q$. Let $z\in \bZ/\gd \bZ$ and $r\in \bZ/\gd \bZ^{\ast}$ such that $\cD'=r\cD+z$. Then the bijection of chambers
\[\go_0:\cT(\cD)_0\to \cT(\cD')_0,\qquad  (x,y)\mapsto (rx+z,ry+z)\] 
clearly preserves $\sim_s$ and $\sim_t$. Thus, $\go_0$ determines an isomorphism $\go:\cT(\cD)\to \cT(\cD')$. Let $G$ and $G'$ be the cyclic groups of order $\gd$ which act on $\cT(\cD)$ and $\cT(\cD')$ respectively, with $g\in G$ and $g'\in G'$ being the chosen generators. Let $\psi:G\to G'$ be the isomorphism such that $g\mapsto (g')^r$. Then $\go$ is $\psi$-equivariant. Thus, equivalent difference sets essentially construct the same universal cover. 

Let $G$ be a cyclic group of order $\gd$ which acts panel-regularly on a generalized triangle $\cT$ of order $q$. Pick a generator $g\in G$, an $s$-panel $P$ of $\cT$, and a $t$-panel $L$ of $\cT$. Then 
\[ \cD=\{  d\in \bZ/\gd\bZ : P \cap g^d\cdot  L\neq \emptyset \} \] 
is a difference set of order $q$, called a difference \emph{obtained} from the action of $G$. Different choices of $g$, $P$ and $L$ produce equivalent difference sets. 

Let $G$ also act on $\cT(\cD)$ by $g\cdot(x,y)=(x+1,y+1)$. For $x\in \bZ/\gd\bZ$ and $d\in \cD$, let $(P_{x},L_{x+d})$ denote the unique chamber of $\cT$ which is contained in $g^{x}\cdot P$ and $g^{x+d}\cdot L$. Then
\[\cT(\cD)\to \cT,\qquad  (x,x+d)\mapsto  (P_{x},L_{x+d})    \]
is an equivariant isomorphism. Conversely, the difference sets obtained from the action of $G$ on $\cT(\cD)$ will be equivalent to $\cD$. Thus, there's essentially a 1-1 correspondence between difference sets up to equivalence and panel-regular actions of cyclic groups on generalized triangles.         

The following is a classical result of Singer, combined with the uniqueness result of Berman.

\begin{thm}[Singer-Berman]          \label{prop:singconjec}
	For all $q$ a prime power, there exists a Desarguesian difference set of order $q$, and this difference set is unique up to equivalence. 
\end{thm}

\begin{proof}
	For existence see \cite{singer1938theorem}, for uniqueness see \cite{berman53finite}.
\end{proof}

As with the Singer cyclic digons, we wish to obtain a canonical realization of $\gd \backslash \cT(\cD)$. Let the set of chambers of $\gd \backslash \cT(\cD)$ be $\cD$, where $d\in \cD$ is the $\pi$-image of $(x ,x+d )\in \cT(\cD)$. To complete our realization of $\gd \backslash \cT(\cD)$, we need to specify a set of defining suites. Let us first determine what the flowers of $\gd \backslash \cT(\cD)$ are by inspecting the covering $\pi:\cT(\cD)\to \gd \backslash \cT(\cD)$.

\begin{prop} \label{prop:flowerfor3} 
	Let $\cD$ be a difference set of order $q$. For each chamber $C\in \cD$, the petals of the flower of $\gd \backslash \cT(\cD)$ based at $C$ are
	\[      [C\xrightarrow{s} x \xrightarrow{t} y\xrightarrow{s} z] \sim [C \xrightarrow{t}  x' \xrightarrow{s} y' \xrightarrow{t} z ]      \]
	 where $C, x, y, z,x',y'\in \cD$, $C\neq x\neq y\neq z$, and $x'-y' = C-x+y-z$.
\end{prop}

\begin{proof}
Let
\[\rho(s,t)=[C\xrightarrow{s} x \xrightarrow{t} y\xrightarrow{s} z]\]
be a maximal $(s,t)$-geodesic of $\gd \backslash \cT(\cD)$ which issues from $C$. Lift $\rho(s,t)$ with respect to $\pi$ to a gallery $\tilde{\rho}(s,t)$ of $\cT(\cD)$. Let $(c,d)\in \bZ/\gd\bZ\times \bZ/\gd\bZ $ be the initial chamber of $\tilde{\rho}(s,t)$. It follows from the construction of $\cT(\cD)$ that
\[\tilde{\rho}(s,t)=\big[(c,d)\xrightarrow{s} (c,d+x-C) \xrightarrow{t} (c+x-y,d+x-C)\xrightarrow{s} (c+x-y,d+x-C+z-y)\big].\]
Let
\[\rho(t,s)=[C\xrightarrow{t} x' \xrightarrow{s} y'\xrightarrow{t} z]\] 
be the unique maximal $(t,s)$-geodesic of $\gd \backslash \cT(\cD)$ with $\rho(s,t)\sim\rho(t,s)$. Let $\tilde{\rho}(t,s)$ be the lifting of $\rho(t,s)$ to a gallery which issues from $(c,d)$. Then 
\[
\tilde{\rho}(t,s)=\big[(c,d)\xrightarrow{t} (c+C-x',d) \xrightarrow{s} (c+C-x',d+y'-x')\xrightarrow{t} (c+C-x'+y'-z,d+y'-x')\big]
.\]
 We have $\tilde{\rho}(s,t)\sim \tilde{\rho}(t,s)$, and so
\[     
(c+x-y,d+x-C+z-y)  = (c+C-x'+y'-z,d+y'-x')    
\] 
 which occurs if and only if
\[  
x'-y' = C-x+y-z  
.\qedhere \]
\end{proof}

Let us assume that $\cD$ is based. Clearly, every difference set is equivalent to a based difference set, so this is no real restriction. We let the defining suites of $\gd \backslash \cT(\cD)$ be those which are induced by the flower based at $0\in \cD$. Thus, by \autoref{prop:flowerfor3}, the defining suites of $\gd \backslash \cT(\cD)$ are the cycles
\[[0\xrightarrow{s} x \xrightarrow{t} y\xrightarrow{s} z \xrightarrow{t}  y' \xrightarrow{s} x' \xrightarrow{t} 0] \]      
 where $x, y, z,x',y'\in \cD$, $0\neq x\neq y\neq z$, and
\begin{equation} \label{eq:3}      \tag{$\clubsuit$}
	y'-x' = x-y+z  .
\end{equation}
Notice that there will be $q^3$ many defining suites. 

\begin{ex}
Let $\cD=\{0,1,3\}$. Then $\cD$ is a difference set of order $2$. The $2^3=8$ defining suites of $7 \backslash \cT(\cD)$ are
\begin{align*}
		[0\xrightarrow{s} & 1\xrightarrow{t} 0\xrightarrow{s} 1   \xrightarrow{t}  3\xrightarrow{t}  1\xrightarrow{s} 0  ]   \\
		[0\xrightarrow{s} & 1\xrightarrow{t} 0\xrightarrow{s} 3\xrightarrow{s} 0   \xrightarrow{t} 3   \xrightarrow{t} 0  ]  \\
		[0\xrightarrow{s} & 1\xrightarrow{t} 3\xrightarrow{s} 0\xrightarrow{t}1\xrightarrow{s} 3\xrightarrow{t}0  ]          \\
		[0\xrightarrow{s} & 1\xrightarrow{t} 3\xrightarrow{s} 1 \xrightarrow{s} 0  \xrightarrow{t} 1  \xrightarrow{t} 0  ]  \\  
		[0\xrightarrow{s} & 3\xrightarrow{t} 0\xrightarrow{s} 1 \xrightarrow{s} 0  \xrightarrow{t} 3 \xrightarrow{t} 0  ]   \\
		[0\xrightarrow{s} & 3\xrightarrow{t} 0\xrightarrow{s} 3 \xrightarrow{s} 0\xrightarrow{t} 1\xrightarrow{t} 0  ]      \\
		[0\xrightarrow{s} & 3\xrightarrow{t} 1\xrightarrow{s} 0 \xrightarrow{s} 3 \xrightarrow{t} 1\xrightarrow{t} 0  ]    \\
		[0\xrightarrow{s} & 3\xrightarrow{t} 1\xrightarrow{s} 3  \xrightarrow{s} 1\xrightarrow{t} 3\xrightarrow{t} 0  ] . 
\end{align*}
\end{ex}

\begin{figure}[t]
	\centering
	\includegraphics[scale=1]{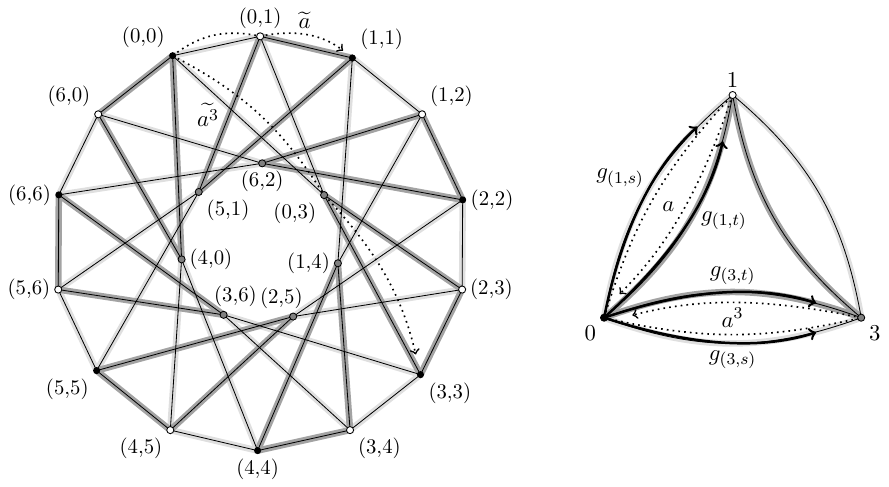}
	\caption{The generalized triangle $\cT(\cD)$ (left) and the Singer cyclic triangle $7 \backslash \cT(\cD)$ (right)}
	\label{fig:fano013}
\end{figure}

Let $\cD$ and $\cD'$ be equivalent difference sets of order $q$ with $\cD'=r\cD+x$. Then it follows from \autoref{prop:flowerfor3} that
\[
\go:\gd \backslash \cT(\cD)\to \gd \backslash \cT(\cD'),\qquad  d\mapsto rd+x
\] 
preserves flowers, and so is an isomorphism. By \autoref{prop:singconjec}, this shows that if \autoref{conj} holds, then for each $q$ a prime power, there is a unique Singer cyclic triangle of order $q$. We let $\gd \backslash \cT(q)$ denote the unique (up to isomorphism) Singer cyclic triangle of order $q$ whose universal cover is $\cT(q)$. Thus, if $\cD$ is Desarguesian, we have $\gd \backslash \cT(\cD)\cong \gd \backslash \cT(q)$. 

The following result shows that given any chamber $C$ in a Singer cyclic triangle $\cP$, we can represent $\cP$ as $ \gd \backslash \cT(\cD)$, for some based difference set $\cD$, such that $C$ is identified with $0\in \cD$.  

\begin{lem} \label{canbe0}
	Let $\cP$ be a Singer cyclic triangle of order $q$, and let $C\in \cP$ be any chamber. Then there exists a based difference set $\cD$ of order $q$ and an isomorphism $\go:\cP\rightarrow \gd \backslash\cT(\cD)$ such that $\go(C)=0$.  
\end{lem}

\begin{proof}
	Let $g$ be a generator of the deck transformation group of the universal cover $\Delta\rightarrow \cP$. Let $\cD'$ be a difference set obtained from this action. Then there exists an isomorphism $\go': \cP \to \gd \backslash \cT(\cD')$. Let $\cD=\cD'-\go'(C)$, and let $\go'$ be the isomorphism
	\[\go': \gd \backslash \cT(\cD')\to \gd \backslash \cT(\cD), \qquad  x\mapsto x-\go'(C).\] 
	 Then we can take $\go=\go''  \circ \go'$. 
\end{proof}

We now calculate the universal group of $\gd \backslash \cT(\cD)$. We assume that $\cD$ is a based difference set. Recall that $\cD^{\ast}=\cD\setminus \{0\}$. First, we need to equip $\gd \backslash \cT(\cD)$ with a generating set $\cS=\cB\sqcup \cI$. Since each panel groupoid of $\gd \backslash \cT(\cD)$ is a setoid, we have $\cB=\emptyset$. For $n\in \cD^{\ast}$ and $\sigma\in \{s,t\}$, let $g_{(n,\sigma)}$ be the edge $0\xrightarrow{\sigma}n$. Then put
\[\cS=\cI=:\big \{    g_{(n,s)} , g_{(n,t)},g_{(n,s)}^{-1} , g_{(n,t)}^{-1}  :  n\in \cD^{\ast}      \big   \}.\] 
 For notational convenience, let $g_{(0,s)}$, $g_{(0,t)}$, $g_{(0,s)}^{-1}$, and $g_{(0,t)}^{-1}$ denote the empty word. It follows from (\ref{eq:3}) that the set of $\cS$-suites of $\gd \backslash \cT(\cD)$ is
\begin{align*}
	\cR=\big\{ g_{(x,s)}g_{(x,t)}^{-1}g_{(y,t)}g_{(y,s)}^{-1}g_{(z,s)} & g_{ (z, t) }^{-1}  g_{(y',t)} g_{ (y', s) }^{-1}  g_{ (x',s)}g_{(x',t)}^{-1}: \\
	& x,y,z,y',x'\in \cD;\ 0\neq x\neq y\neq z;\ y'-x'=x-y+z       \big \}  . 
\end{align*}
We now substitute $a_{(n)}=g_{(n,s)} g_{(n,t)}^{-1}$, for $n\in \cD^{\ast}$. Thus, the new generating set is
\[\cS':=\big \{    g_{(n,s)} , g_{(n,t)},g_{(n,s)}^{-1} , g_{(n,t)}^{-1},a_{(n)}  :  n\in \cD^{\ast}    \big      \}\] 
and, letting $a_{(0)}$ and $a_{(0)}^{-1}$ denote the empty word, a new set of equivalent relations is
\begin{align*}
	\cR':= \big\{ a_{(x)}a_{(y)}^{-1}a_{(z)}a_{(y')}^{-1} a_{(x')}=1  &: x,y,z,y',x'\in \cD;\ 0\neq x\neq y\neq z;\ y'-x'=x-y+z \\
	a_{(n)}=g_{(n,s)} g_{(n,t)}^{-1}   &: n\in \cD^{\ast} \big \} .
\end{align*}
Let us rearrange $a_{(x)}a_{(y)}^{-1}a_{(z)}a_{(y')}^{-1} a_{(x')}=1$ to give
\begin{equation*} \label{eqqq} \tag{$\spadesuit$}
	a_{(x)}a_{(y)}^{-1}a_{(z)}= a_{(x')}^{-1} a_{(y')}
\end{equation*}     
for $x,y,z,y',x'\in \cD$, $0\neq x\neq y\neq z$, and $y'-x'=x-y+z$. Let $e,e'\in \cD$ be the unique integers such that $e-e'=1$, and put $a=a_{(e)} a_{(e')}^{-1}$. In particular, if $1\in \cD$, we just have $a=a_{(1)}=g_{(1,s)} g_{(1,t)}^{-1}$. Fix $n\in \{1,\dots, \gd-1   \}$, and let $c,c'\in \cD$ be the unique integers such that
\[
c-c'=n.
\] 
If $n\neq \gd-1$, let $d,d'\in \cD$ be the unique integers such that 
\[
d-d'=n+1.
\] 
If $n\neq -e$, let $f,f'\in \cD$ be the unique integers such that
\[
f'-f=n+e.
\] 
Now, if $n=\gd-1$, then $c'=e$ and $c=e'$, and so, 
\begin{equation*}  \label{eqqqq} \tag{$\varheart$}
	a_{(c)}a_{(c')}^{-1}a=        a_{(c)} a_{(c')}^{-1} a_{(e)} a_{(e')}^{-1}=1.
\end{equation*}
We now claim that if $n\neq \gd-1$, then
\[   
a_{(c)}a_{(c')}^{-1}a      =a_{(d)} a_{(d')}^{-1} .             
\] 
If $n\neq \gd-1$ but $n= -e$, then $c=0$, $c'=e$, $d=0$, and $e'=d'$, and so
\[   
a_{(c)}a_{(c')}^{-1}a=        a_{(c)} a_{(c')}^{-1} a_{(e)} a_{(e')}^{-1}=        a_{(e')}^{-1}              =a_{(d)} a_{(d')}^{-1} .             
\]
If $n\neq \gd-1$, $n\neq -e$, $c'\neq e$, and $f'\neq e'$, then by two applications of (\ref{eqqq}), we have
\[   
a_{(c)}a_{(c')}^{-1}a=        a_{(c)} a_{(c')}^{-1} a_{(e)} a_{(e')}^{-1}=    a_{(f)}^{-1} a_{(f')}                                   a_{(e')}^{-1}              =a_{(d)} a_{(d')}^{-1} .             
\]
If $n\neq \gd-1$, $n\neq -e$, and $c'=e$, then $c=d$ and $e'-d'$, and so
\[   
a_{(c)}a_{(c')}^{-1}a=        a_{(c)} a_{(c')}^{-1} a_{(e)} a_{(e')}^{-1}=   a_{(c)}  a_{(e')}^{-1}              =a_{(d)} a_{(d')}^{-1} .             
\]
Finally, if $n\neq \gd-1$, $n\neq -e$, and $f'=e'$, then $d=0$ and $f=d'$, and so
\[   
a_{(c)}a_{(c')}^{-1}a=        a_{(c)} a_{(c')}^{-1} a_{(e)} a_{(e')}^{-1}=    a_{(f)}^{-1} a_{(f')}                                   a_{(e')}^{-1}              =a_{(d)} a_{(d')}^{-1} .             
\]
This proves the claim. Therefore, by induction, we have $a^n=a_{(c)}a_{(c')}^{-1}$, and in particular $a^n=a_{(n)}$ if $n\in \cD$. The fact that $a^\gd=1$ then follows from (\ref{eqqqq}). Notice that in the presence of $a^n=a_{(n)}$ and $a^\gd=1$, the relations of the form $a_{(x)}a_{(y)}^{-1}a_{(z)}a_{(y')}^{-1} a_{(x')}=1$ are redundant. Thus a new generating set is
\[\cS'':=\big \{    g_{(n,s)} , g_{(n,t)},g_{(n,s)}^{-1} , g_{(n,t)}^{-1},a  :  n\in \cD^{\ast}     \big    \}\] 
with relations
\[   \cR'':= \big\{  a^{\gd}=1,\ a^n=g_{(n,s)} g_{(n,t)}^{-1}   : n\in \cD^{\ast}   \big \} .\] 
We obtain the following.

\begin{lem} \label{universalgruop1}
Let $\cD$ be a based difference set and let $\gd \backslash \cT(\cD)$ be the Singer cyclic triangle constructed from $\cD$. Then the universal group of $\gd \backslash \cT(\cD)$ is
\[
\FG(\gd \backslash \cT(\cD))=   
\big \la      
g_{(n,\sigma)} ,g_{(n,\sigma)}^{-1}, a  \ | \   a^{\gd}=1, \   a^n=g_{(n,s)} g_{(n,t)}^{-1},\ g_{(n,\sigma)}g_{(n,\sigma)}^{-1}=1   
\big \ra  
\]
for $\sigma\in \{s,t\}$ and $n\in \cD^{\ast}$.  
\end{lem}

Let $T$ be the spanning tree
\[T:=\big \{g_{(n,t)}, g_{(n,t)} ^{-1}: n\in \cD^{\ast}\big  \}.\] 
Thus, we recover the fundamental group of $\gd \backslash \cT(\cD)$,
\[\pi_1(\gd \backslash \cT(\cD),T)  =  \big\la   g_{(n,s)} ,g_{(n,s)}^{-1}, a \ | \          
a^{\gd}=1,  \     a^n=g_{(n,s)} ,\ g_{(n,s)}g_{(n,s)}^{-1}=1
\big\ra\sim \big \la  a \ | \ a^{\gd}=1 \big\ra.\] 
Notice that the image of the gallery $[ 0\xrightarrow{s} n \xrightarrow{t} 0   ]$ in the fundamental group at $T$ is $a^n$, see \autoref{fig:fano013}. 

\section{Singer Cyclic Lattices of Type $M$} \label{sec:Singer Cyclic Lattices of Type $M$}

We now construct the Singer cyclic lattices of type $M$, where $m_{st}\in\{2,3,\infty\}$ for all distinct $s,t\in S$, and the defining graph of $M$ is connected. Modulo an equivalence relation on our construction, this classifies all such lattices. This equivalence is described in the $\wt{A}_2$ case in terms of `based difference matrices' by Witzel \cite{witzel2016panel}, which builds on the work of Essert \cite{essert2013geometric}. 

\subsection{Gluing Matrices $\cM$}

Let $M=(m_{st})_{s,t\in S}$ be a Coxeter matrix on a set $S$ such that $m_{st}\in\{2,3,\infty\}$ for all $s,t\in S$, $s\neq t$, and whose defining graph $L$ is connected. Gluing matrices will play the same role as based difference matrices in the work of Essert. Let $q\in \bZ_{\geq 2}$ and put $k:=q+1$. Let 
\[
\cC:=\bZ/k\bZ\qquad  \text{and} \qquad  \cC^{\ast}:=\{ 1,\dots,q \}\subset \cC
.\] 
Let $\bar{E}$ be an orientation of $L$; formally, let $\bar{E}\subseteq S\times S$ such that for all $(s,t)\in \bar{E}$, we have $\{s,t\}\in E(L)$, and for all $\{s,t\}\in E(L)$, exactly one of $\{(s,t),(t,s)\}$ is contained in $\bar{E}$. We define a \emph{gluing matrix} $\cM=\cM_{M,q}$ of type $M$ and order $q$ to be a matrix
\[          \cM:  \cC^{\ast} \times      \bar{E}\to \bZ_{\geq 1}    \]
such that
\begin{enumerate}[label=(\roman*)]
\item
for $(s,t)\in \bar{E}$ with $m_{st}=2$, each $\cC^{\ast}$-tuple $\cM(  -,(s,t) )$ is a permutation of $\cC^{\ast}$ 
\item
for $(s,t)\in \bar{E}$ with $m_{st}=3$,  each $\cC^{\ast}$-tuple $\cM(  -,(s,t) )$ is a permutation of some $\cD^{\ast}$, where $\cD$ is a based difference set of order $q$.
\end{enumerate}
Let us denote $\cM(n,(s,t))$ by $n(st)$. If $m_{st}=2$, we think of the integer $n(st)$ as being an element of $\bZ/k\bZ$, and if $m_{st}=3$, we think of the integer $n(st)$ as being an element of $\bZ/\gd\bZ$.  

\begin{figure}[t]
\centering
\includegraphics[scale=1]{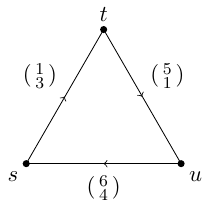}
\caption{A gluing matrix of type $\wt{A}_2$ and order $2$}
\label{fig:glueingdataA22undir}
\end{figure}

\begin{ex}
Let $L$ be the defining graph of $\wt{A}_2$. \autoref{fig:glueingdataA22undir} shows $L$ decorated with the data of the following gluing matrix $\cM=\cM_{L,2}$,
\begin{center}
\begin{tabular}{ c| c c c }
& $(s,t)$ & $(t,u)$ & $(u,s)$ \\ 
\hline
$1$ & $1$ & $5$ & $6$ \\  
$2$ & $3$ & $1$ & $4$    
\end{tabular}
\end{center}
Notice that each column is of the form $\cD^{\ast}$ for some based difference set $\cD$ of order $2$.
\end{ex}

\subsection{The Singer Graph $\cW_\cM$} 

Let $L$, $\bar{E}$, $q$, $k$, $\cC$, and $\cC^{\ast}$ be as above, and let $\cM$ be a gluing matrix of order $q$. We associated a Singer graph $\cW=\cW_\cM$ of type $M$ to $\cM$ as follows. The set of chambers is $\cW_0=\cC$, and each panel groupoid is $\cW_s\cong  1\times k$, for $s\in S$. Let $(s,t)\in \bar{E}$, and let $J=\{s,t\}$. Notice that there is exactly one $J$-residue $\cW_J$ of $\cW$. 

If $m_{st}=2$, let $\Omega_{st}:\cW_J \to k\backslash \mathbf{D}(q)$ be the chamber system morphism such that for $x\in \cW_0$, we have
\[   
\Omega_{st}(x):= 
\begin{cases}
0 &\quad\text{if}\   x=0 \\
x(st) &\quad\text{otherwise.} 
\end{cases}
\]
Then, let the defining suites of $\cW_{J}$ be the $\Omega_{st}$-preimages of the defining suites of $k\backslash \mathbf{D}(q)$. Thus, the defining suites are
\[
 [0 \xrightarrow{s} y\xrightarrow{t} z \xrightarrow{s} y'\xrightarrow{t} 0  ]    
\] 
where $y,z,y'\in\cC$, $0\neq y\neq z$, and $\Omega_{st}(y')=\Omega_{st}(z)-\Omega_{st}(y)$. 

If $m_{st}=3$, put $\cD=\cM(  -,(s,t) )\cup \{0\}$, and let $\Omega_{st}:\cW_J \to \gd\backslash \cT(\cD)$ be the chamber system morphism such that for $x\in \cW_0$, we have
\[   
\Omega_{st}(x):= 
\begin{cases}
0 &\quad\text{if}\   x=0 \\
x(st) &\quad\text{otherwise.} 
\end{cases}
\]
Then, let the defining suites of $\cW_J$ be the $\Omega_{st}$-preimages of the defining suites of $\gd\backslash \cT(\cD)$. Thus, the defining suites are
\[ [0 \xrightarrow{s} x\xrightarrow{t} y \xrightarrow{s} z\xrightarrow{t} y' \xrightarrow{s} x' \xrightarrow{t} 0  ]    \] 
where $y,z,y',x'\in\cC$, $0\neq x \neq  y\neq z$, and 
\[\Omega_{st}(y')-\Omega_{st}(x')=\Omega_{st}(x)-\Omega_{st}(y)+\Omega_{st}(z).\] 
This defines the Weyl data $\cW$. Then $\cW$ is a Weyl graph by the generalization of Tits's local-to-global result \cite[Theorem 5.13]{nor2}.

Let $\cW=\cW_\cM$ be as constructed above. We now calculate the universal group of $\cW$. First, we need a generating set $\cS=\cB \sqcup \cI$. Since each panel groupoid of $\cW$ is a setoid, we have $\cB=\emptyset$. For $n\in \cC^{\ast}$ and $\sigma\in S$, let $g_{(n,\sigma)}$ be the edge $0\xrightarrow{\sigma}n$. Then put
\[\cS=\cI:=\big \{    g_{(n,\gs)},g_{(n,\gs)}^{-1}  :  n\in \cC^{\ast},  \gs\in S     \big   \}.\]
Then the set of $\cS$-suites $\cR$ of $\cW$ is the union of the $\cS$-suites in each of its spherical $2$-residues,
\begin{align*}
\cR= \big   \{\ &  g_{(y,s)}  g_{(y,t)}^{-1} g_{(z,t)} g_{ (z, s) }^{-1} g_{ (y',s)}   g_{(z-y,t)}^{-1}  \\
& g_{(x,u)}  g_{(x,v)}^{-1}g_{(y,v)}g_{(y,u)}^{-1}g_{(z,u)}  g_{ (z, v) }^{-1}  g_{(y',v)} g_{ (y', u) }^{-1}  g_{ (x',u)}g_{(x',v)}^{-1}\ \big   \} 
\end{align*}
where $(s,t)\in \bar{E}$ with $m_{st}=2$, $y,z\in  \cC$ with $0\neq y \neq z$, and 
\[\Omega_{st}(y')=\Omega_{st}(z)-\Omega_{st}(y)\] 
and $(u,v)\in \bar{E}$ with $m_{uv}=3$, $x,y,z,y',x'\in \cC$ with $0\neq x\neq y\neq z$, and
\[\Omega_{st}(y')-\Omega_{st}(x')=\Omega_{st}(x)-\Omega_{st}(y)+\Omega_{st}(z).\]
We now make the same substitutions in each spherical $2$-residue of $\cW$ that we made in order to obtain the presentations of \autoref{universalgruop2} and \autoref{universalgruop1}. For $(s,t)\in \bar{E}$ with $m_{st}=2$, let $n\in \cC^{\ast}$ such that $\Omega_{st}(n)=1$. Then put
\[a_{st}= g_{(n,s)} g_{(n,t)}^{-1}.\] 
For $(s,t)\in \bar{E}$ with $m_{st}=3$, let $n,n'\in \cC$ such that $\Omega_{st}(n)- \Omega_{st}(n')=1$. For notational convenience, let $g_{(0,\gs)}$ denote the empty word for all $\sigma\in S$, and put
\[
a_{st}=    g_{(n,s)} g_{(n,t)}^{-1} (g_{(n',s)} g_{(n',t)}^{-1})^{-1}          .
\]
Thus, for both $m_{st}=2$ and $m_{st}=3$, $a_{st}$ is the `$a$' from the calculation of the universal group of the target of $\Omega_{st}$. Let us also stop including the $g_{(n,s)}^{-1}$ as generators for simplicity, which is obviously possible. Then the new set of generators is
\[   \cS':=\big \{    g_{(n,\gs)},a_{st}  :  n\in \cC^{\ast};\  \gs\in S;\ (s,t)\in \bar{E}     \big   \}\]
and it follows from \autoref{universalgruop2} and \autoref{universalgruop1} that a new set of equivalent relations is
\[
\cR':= \big \{  (a_{st})^{\gd(st)}=1        ,\  (a_{st})^{n(st)}=g_{(n,s)} g_{(n,t)}^{-1}   :  (s,t)\in \bar{E};\ n\in \cC^{\ast}  \big \}
\]
where for $(s,t)\in \bar{E}$, we have
\[   
\gd(st) = 
\begin{cases}
k=q+1 &\quad\text{if}\  m_{st}=2 \\
\gd=q^{2}+q+1 &\quad\text{if}\  m_{st}=3.  
\end{cases}
\]
Thus, we obtain the following.

\begin{lem} \label{universalgruop1}
Let $\cM$ be a gluing matrix and let $\cW=\cW_{\cM}$ be the Weyl graph associated to $\cM$. Then the universal group of $\cW$ is
\[
\FG(\cW)=   \big \la     g_{(n,\gs)},a_{st}     \ |  \  
(a_{st})^{\gd(st)}=1        ,\  (a_{st})^{n(st)}=g_{(n,s)} g_{(n,t)}^{-1}      \big \ra  
\]
for  $n\in \cC^{\ast}$, $\gs\in S$, and $(s,t)\in \bar{E}$. 
\end{lem}

We define an \emph{$L$-cycle} to be a sequence $\gs_1,\dots,\gs_{\mu}$ of adjacent vertices of $L$ with $\gs_1=\gs_{\mu}$. We denote the set of $L$-cycles in $L$ by $\Cir(L)$. If $(s,t)\notin \bar{E}$, let $a_{st}=a_{ts}^{-1}\in \FG(\cW)$. Then for each $L$-cycle $\gs_1,\dots,\gs_{\mu}$, and for each $n\in \cC^{\ast}$, we have
\begin{align*}
&(a_{\gs_1\gs_2})^{n(\gs_1\gs_2)} (a_{\gs_2\gs_3})^{n(\gs_2\gs_3)} \dots (a_{\gs_{\mu-1}\gs_{\mu}})^{n(\gs_{\mu-1}\gs_{\mu})} \\
=\ & g_{(n,\gs_1)} g_{(n,\gs_2)}^{-1} g_{(n,\gs_2)} g_{(n,\gs_3)}^{-1}  \dots      g_{(n,\gs_{\mu-1})} g_{(n,\gs_{\mu})}^{-1} && \text{since } (a_{st})^{n(st)}=g_{(n,s)} g_{(n,t)}^{-1}\\
=\ & 1 && \text{since } \gs_1=\gs_{\mu}.
\end{align*}
Therefore $(a_{\gs_1\gs_2})^{n(\gs_1\gs_2)} (a_{\gs_2\gs_3})^{n(\gs_2\gs_3)} \dots (a_{\gs_{\mu-1}\gs_{\mu}})^{n(\gs_{\mu-1}\gs_{\mu})}=1$ is a consequence of the relations of $\FG(\cW)$, which we call the relation \emph{induced} by $\gs_1,\dots,\gs_{\mu}$. 

\subsection{The Singer Lattice $\Gamma_\cM$.} 

We now obtain a presentation of the fundamental group of $\cW=\cW_{\cM}$ whose generators are the $a_{st}$'s. Fix $\xi \in S$, and let us assume that $\bar{E}$ has been chosen such that for all $u\in S$, if $\{u,\xi\}\in E(L)$, then $(u,\xi)\in \bar{E}$, i.e. the oriented edges at $\xi$ terminate at $\xi$. Let $T$ be the spanning tree 
\[
T:=\big \{ g_{(n,\xi)}, g_{(n,\xi)}^{-1}: n\in \cC^{\ast} \big  \}
.\] 
Then by quotienting out $T$ in $\FG(\cW)$, we obtain
\begin{align*}
	\pi_1(\cW,T)=   \big \la     g_{(n,\gs)},a_{st}, & a_{u \xi}     \ |  \  \\
	(a_{st})^{\gd(st)} & =(a_{u \xi})^{\gd(u \xi)}=1        ,\  (a_{st})^{n(st)}=g_{(n,s)} g_{(n,t)}^{-1},\  (a_{u \xi})^{n(u \xi)}=g_{(n,u)}         \big \ra  
\end{align*}
for  $n\in \cC^{\ast}$, $\gs\in S\setminus \{ \xi \}$, $(s,t)\in \bar{E}$ with $s,t\neq \xi$, and $(u,\xi)\in \bar{E}(L)$.  

For $s,t\in S$, if $(s,t)\notin \bar{E}$, let $a_{st}=a_{ts}^{-1}\in \pi_1(\cW,T)$. For $\gs\in S$, let $\gs,\gs_1,\dots,\gs_{\mu},\xi$ be a sequence of adjacent vertices of $L$ (recall that $L$ is connected). For $n\in \cC^{\ast}$, we have $(a_{ \gs_{\mu} \xi })^{n(\gs_{\mu} \xi )}=g_{(n,\gs_\mu)}$, and so
\begin{align*} 
g_{(n,\gs)}   = \   &   g_{(n,\gs)}     g_{(n,\gs_1)}^{-1} g_{(n,\gs_1)}     \dots    g_{(n,\gs_\mu)}^{-1}  g_{(n,\gs_\mu)}      \\        
= \ & (a_{\gs \gs_1})^{n(\gs \gs_1)} (a_{\gs_1 \gs_2})^{n(  \gs_1 \gs_2)}\dots   (a_{\gs_{\mu-1} \gs_{\mu}})^{n(  \gs_{\mu-1} \gs_{\mu})}    (a_{ \gs_{\mu} \xi })^{n(\gs_{\mu} \xi )}      .   
\end{align*}
In the presence of relations induced by $L$-cycles, the relations of the form $(a_{st})^{n(st)}=g_{(n,s)} g_{(n,t)}^{-1}$ are redundant since, after substitution, the relation is induced by an $L$-cycle of the form
\[  \xi, \gs_{\mu},\dots, \gs_1,t,      s,\gs'_1,\dots,\gs'_{\mu},\xi.\]
Therefore we can drop the $g_{(n,\gs)}$ from the generating set of $\pi_1(\cW,T)$ and include the relations induced by $L$-cycles to obtain an equivalent group presentation of $\pi_1(\cW,T)$.

\begin{thm} \label{mainmain}
Let $\cM$ be a gluing matrix and let $\cW=\cW_{\cM}$ be the Weyl graph associated to $\cM$. Then the fundamental group $\Gamma=\Gamma_{\cM}$ of $\cW$ has the presentation
\[\Gamma=   \big \la    a_{st},a_{ts}   \ |\  a_{st}a_{ts}  = (a_{st})^{\gd(st)} =(a_{\gs_1\gs_2})^{n(\gs_1\gs_2)}\dots(a_{\gs_{\mu-1}\gs_{\mu}})^{n(\gs_{\mu-1}\gs_{\mu})}=1         \big \ra  \]
for $(s,t)\in \bar{E}$, $\ n\in \cC^{\ast}$, and $\gs_1,\dots,\gs_{\mu}\in \Cir(L)$. Thus, $\Gamma$ is a Singer cyclic lattice in a building $\Delta=\Delta_{\cM}$ of type $M$, with $\Gamma\backslash \Delta\cong \cW$.  
\end{thm}

The number of relations induced by $L$-cycles that need to be included in order to obtain a sufficient set of relations will depend upon the homotopy type of $L$.

\subsection{Classification of Singer Cyclic Lattices of type $M$}

Let $M$ be a Coxeter matrix with $m_{st}\in\{2,3,\infty\}$ for all distinct $s,t\in S$, whose defining graph $L$ is connected. We now show that every Singer cyclic lattice 
\[\gG< \Aut(\Delta)\] 
in a building $\Delta$ of type $M$ is equivalent to a Singer cyclic lattice of the form $\gG_\cM< \Aut(\Delta_\cM)$, for some gluing matrix $\cM$ of type $M$.   

Let $\Delta$ be a locally finite building of type $M$, and let $\Gamma< \Aut(\Delta)$ be a Singer cyclic lattice in $\Delta$ of order $q$. Then $\Gamma\backslash \Delta$ is a Singer graph of order $q$ with spherical Weyl polygons either the digon $k \backslash \mathbf{D}(q)$, the triangle $\gd \backslash \cT(q)$, or (if they exist) non-Desarguesian Singer cyclic triangles. Therefore the only choices $\Gamma\backslash \Delta$ has are in how these polygons are glued together. We now show that gluing matrices encode all the possible choices, and so $\gG\backslash \Delta  \cong  \cW_{\cM}$ for some gluing matrix $\cM$.    

Let $\cC= \bZ/k\bZ$, and let $\bar{E}$ be an orientation of $L$. By choosing a bijection $(\Gamma\backslash \Delta)_0\to \cC$, let us identify the chambers of $\Gamma\backslash \Delta$ with $\cC$. We now define a gluing matrix 
\[\cM=\cM(\gG):  \cC^{\ast} \times      \bar{E}\to \bZ_{\geq 1}\] 
of type $M$ and order $q$ as follows. For each $(s,t)\in \bar{E}$ such that $m_{st}=2$, by \autoref{cor:rotatingdigon}, there exists a permutation $\cC\to \cC$ with $0\mapsto 0$ which is an isomorphism on the $\{s,t\}$-residue of $\Gamma\backslash \Delta$ into $k\backslash \mathbf{D}(q)$. Let $\cM(-,(s,t))$ be the restriction of the permutation $\cC\to \cC$ to $\cC^{\ast}$. For each $(s,t)\in \bar{E}$ such that $m_{st}=3$, by \autoref{canbe0}, there exists a based difference set $\cD$ and a bijection $\cC\to \cD$ with $0\mapsto 0$ which is an isomorphism on the $\{s,t\}$-residue of $\Gamma\backslash \Delta$ into $\gd\backslash \cT(\cD)$. Let $\cM(-,(s,t))$ be the restriction of $\cC\to \cD$ to $\cC^{\ast}$. We call $\cM$ the gluing matrix \emph{associated} to $\Gamma$. Then we obtain the following.

\begin{thm}
Let $M$ be a Coxeter matrix with $m_{st}\in\{2,3,\infty\}$ for all $s,t\in S$, $s\neq t$, whose defining graph is connected. Let $\Delta$ be a locally finite building of type $M$, and let $\Gamma< \Aut(\Delta)$ be a Singer cyclic lattice in $\Delta$. Let $\cM=\cM(\gG)$ be a gluing matrix associated to $\Gamma$. Then there exists an isomorphism $\go:\gG\backslash \Delta   \to   \cW_{\cM} $.
\end{thm}

\begin{proof}
Let $\go:\gG\backslash \Delta   \to   \cW_{\cM} $ be the morphism whose map on chambers is the identity $\cC \to \cC$. The fact that $\go$ is an isomorphism then follows directly from the definition of $\cM$ and the construction of $\cW_{\cM}$.
\end{proof}

This gives a classification of Singer cyclic lattices of type $M$ modulo an equivalence relation on gluing matrices, since different gluing matrices may construct the same lattice. This equivalence for $M=\wt{A}_2$ is described in \cite[Corollary 3.19]{witzel2016panel} in terms of difference matrices. 

\section{Examples of Singer Cyclic Lattices} \label{sec:Examples of Singer Cyclic Lattices}

\begin{figure}[t]
	\centering
	\begin{minipage}{.5\textwidth}
		\centering
		\includegraphics[scale=1]{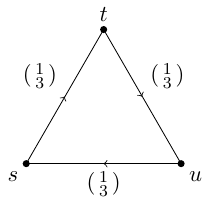}
		\caption{The gluing matrix $\cM$}
		\label{fig:glueingdataA22}
	\end{minipage}%
	\begin{minipage}{.5\textwidth}
		\centering
		\includegraphics[scale=1]{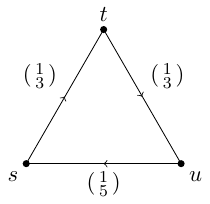}
		\caption{The gluing matrix $\cM'$}
		\label{fig:glueingdataA222}
	\end{minipage}
\end{figure}

\subsection{Singer Cyclic Lattices of Type $\wt{A}_2$}

The Singer cyclic lattices of type $\wt{A}_2$ were constructed in \cite{essert2013geometric}, and classified in \cite{witzel2016panel}. In \cite{witzel2016panel}, it is shown that there are two non-isomorphic Singer lattices of type $\wt{A}_2$ and order $2$. Let us construct these lattices within our framework of Weyl graphs:
\begin{enumerate}[label=(\arabic*)]
	\item
	Take the gluing matrix $\cM$ of type $\wt{A}_2$ shown in \autoref{fig:glueingdataA22}. Put $a=a_{st}$, $b=a_{tu}$, and $c=a_{us}$. Notice that we only need to include one $L$-cycle in the presentation. From \autoref{mainmain}, we obtain
	\[    \Gamma_{\cM}=  \big \la a, b, c\ |\  a^7= b^7= c^7=abc=a^3b^3c^3=1  \big \ra  .     \]
	\item
	Take the gluing matrix $\cM'$ of type $\wt{A}_2$ shown in \autoref{fig:glueingdataA222}. Put $a=a_{st}$, $b=a_{tu}$, and $c=a_{us}$. Then
	\[    \Gamma_{\cM}=  \big \la a, b, c\ |\  a^7= b^7= c^7=abc=a^3b^3c^5=1  \big \ra    .   \]
\end{enumerate}

\subsection{Hyperbolic Singer Cyclic Lattices}

We now construct some hyperbolic examples, i.e. where the Davis realization of an apartment has a natural $\CAT(-1)$ metric which is a hyperbolic space.

\begin{ex}
Let $M$ be the Coxeter matrix on $S=\{s,t,u,v\}$ with 
\[
m_{st}=m_{tu}=m_{uv}=m_{vs}=3\qquad  \text{and}\qquad m_{su}=m_{tv}=\infty
.\] 
The Coxeter group $W$ of type $M$ acts on the hyperbolic plane by tiling with squares whose internal angles are all $\pi/3$. Let $\cM$ be the following gluing matrix of type $M$ and order $3$,
\begin{center}
\begin{tabular}{ c| c c c c }
& $(s,t)$ & $(t,u)$ & $(u,v)$ & $(v,s)$ \\ 
\hline
$1$ & $1$ & $1$ & $1$ & $1$  \\  
$2$ & $3$ & $3$ & $3$ & $3$  \\  
$3$ & $9$ & $9$ & $9$ & $9$ 
\end{tabular}
\end{center} 
Put $a=a_{st}$, $b=a_{tu}$, $c=a_{uv}$, and $d=a_{vs}$. We only need to include one $L$-cycle in the presentation. We obtain
\[    
\Gamma_\cM=  \big \la a, b, c, d\ |\ a^{13}= b^{13}= c^{13}=d^{13}=abcd=a^3b^3c^3d^3  = a^9b^9c^9d^9    =1  \big \ra    
.\]
Then $\Delta_{\cM}$ is a Fuchsian building. In the language of \cite{anneprobsonpolycomp}, $\Gamma_\cM$ is a uniform lattice in the Davis realization of $\Delta_{\cM}$, which is a $(4,L)$-complex, where $L$ is the simplicial building of the projective plane $\PG(2,3)$.
\end{ex}

\begin{ex}
Let $M$ be the Coxeter matrix on $S=\{s,t,u,v\}$ with 
\[
m_{st}=m_{vs}=2, \qquad   m_{tu}=m_{uv}= 3,\qquad m_{su}=m_{tv}=\infty
.\] 
Let $\cM$ be the following gluing matrix of type $M$,
\begin{center}
\begin{tabular}{ c| c c c c }
& $(s,t)$ & $(t,u)$ & $(u,v)$ & $(v,s)$ \\ 
\hline
$1$ & $1$ & $1$ & $1$ & $1$  \\  
$2$ & $2$ & $3$ & $3$ & $2$  \\  
$3$ & $3$ & $9$ & $9$ & $3$ 
\end{tabular}
\end{center}
Put $a=a_{st}$, $b=a_{tu}$, $c=a_{uv}$, and $d=a_{vs}$. Then
\[   \Gamma_\cM=  \big\la     a,b,c,d\ |\      a^4=  b^{13}= c^{13}= d^4  =   abcd =a^2 b^3  c^3  d^2 = a^3 b^9  c^9  d^3 =1   \big\ra .\]
This is a lattice in a Fuchsian building whose Davis realization models apartments as copies of the hyperbolic plane tessellated by Saccheri quadrangles. The links of the Davis realization are the complete bipartite graph on $4+4$ vertices, and the simplicial building of the projective plane $\PG(2,3)$.
\end{ex}

\subsection{Wild Singer Cyclic Lattices}

We finish by constructing two wild Singer cyclic lattices. 

\begin{ex} \label{exx1}
Let $M$ be the Coxeter matrix on $S=\{s,t,u,v\}$ with 
\[
m_{st}=m_{tu}=m_{us}=m_{uv}=3\qquad  \text{and} \qquad  m_{sv}=m_{tv}=\infty
.\] 
Take the gluing matrix $\cM$ of type $M$ and order $2$ shown in \autoref{fig:glueingdata}. Put $a=a_{st}$, $b=a_{tu}$, $c=a_{us}$, and $d=a_{uv}$. We only need to include one $L$-cycle in the presentation. We have
\[    
\Gamma_{\cM}=  \big \la a,b,c,d\ |\ a^7= b^7= c^7=d^7= abc=a^5b^3c^3=1  \big \ra    
.\]
\end{ex}

\begin{figure}[t]
	\centering
	\begin{minipage}{.5\textwidth}
		\centering
		\includegraphics[scale=1]{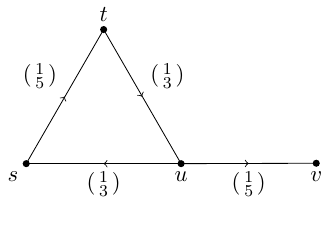}
		\caption[The gluing matrix $\cM$ of \autoref{exx1}]{The gluing matrix $\cM$ of \autoref{exx1}} 
		\label{fig:glueingdata}
	\end{minipage}%
	\begin{minipage}{.5\textwidth}
		\centering
		\includegraphics[scale=1]{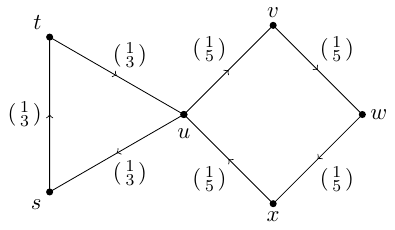}
		\caption[The gluing matrix $\cM$ of \autoref{exx2}]{The gluing matrix $\cM$ of \autoref{exx2}} 
		\label{fig:gluingmat2cyc}
	\end{minipage}
\end{figure}

\begin{ex} \label{exx2}
Let $M$ be the Coxeter matrix whose defining graph consists of a $3$-cycle and a $4$-cycle identified at a vertex, and whose labels are all $3$. Take the gluing matrix $\cM$ of type $M$ and of order $2$ shown in \autoref{fig:gluingmat2cyc}. Put $a=a_{st}$, $b=a_{tu}$, $c=a_{us}$, $d=a_{uv}$, $e=a_{vw}$, $f=a_{wx}$, and $g=a_{xu}$. The $3$-cycle and the $4$-cycle of $L$ are sufficient $L$-cycles for the presentation, thus 
\[   
\Gamma_{\cM}=  \big \la a,b,c,d,e,f,g\ |\ a^7, b^7, c^7,d^7,e^7,f^7,g^7, abc,a^3b^3c^3,defg,d^5e^5f^5g^5  \big \ra    
.\]
\end{ex}

\bibliographystyle{alpha}
\bibliography{sample}


\end{document}